\documentclass[12pt, a4]{amsart}
\usepackage{amsgen,amsmath,amstext,amsbsy,amsopn,amsfonts,amssymb}
\usepackage{amsmath,amssymb,amsfonts,amsthm,enumerate}
\usepackage{amsthm}
\usepackage[utf8]{inputenc}
\usepackage{color}
\usepackage[pagebackref]{hyperref}
\usepackage[alphabetic]{amsrefs}
\usepackage{xcolor}

\newcommand*{\magenta}[1]{\textcolor{magenta}{#1}}

\newtheorem{theorem}{Theorem}
\newtheorem{lemma}[theorem]{Lemma}

\newtheorem{corollary}[theorem]{Corollary}
\newtheorem{proposition}[theorem]{Proposition}
\newtheorem{obs}[theorem]{Observation} \newtheorem{defi}[theorem]{Definition}

\newenvironment{definition}{\begin{defi}\rm}{\end{defi}}
\newtheorem{exa}[theorem]{Example}
\newenvironment{example}{\begin{exa}\rm}{\end{exa}}
\newtheorem{rem}[theorem]{Remark}
\newenvironment{remark}{\begin{rem}\rm}{\end{rem}}
\newtheorem{rems}[theorem]{Remarks}

\newtheorem{ack}[theorem]{Acknowlegment}

\def\H{\mathcal H}
\def\HS{\mathroman{HS}}
\def\L{\mathcal L}
\def\K{\mathcal K}

\def\M{\mathcal M}
\def\F{\mathcal F}

\def\ZZ{{\mathbf Z}}

\def\RRR{{\mathbf R}}

\def\FF{\mathbf F}

\def\RR+{{\mathbf R}^*}

\def\KK{\mathbf K}
\def\kk{\mathbf k}
\def\GG{\mathbf G}

\def\Un{\mathbf 1}
\def\HS{\rm HS}
\def\Q_p{{\mathbf Q}_p}

\def\ind{{\rm Ind}}

\def\eps{\varepsilon}

\def\vfi{\varphi}

\def\Aut{{\rm Aut}}

\def\ind{{\rm Ind}}

\def\Ad{{\rm Ad}}

\def\Irr{{\mathrm {Irr}}}

\newcommand{\Ind}{\operatorname{Ind}}

\def\tout{\qquad\text{for all}\quad}
\newcommand{\N}{\mathbf N}

\newcommand{\UU}{\mathbf U}
\newcommand{\LL}{\mathbf L}
\newcommand{\MM}{\mathbf M}

\begin{document}

\title[Unitary representations of algebraic groups]{On  unitary representations of algebraic groups over local fields}

\address{Bachir Bekka \\ Univ Rennes \\ CNRS, IRMAR--UMR 6625\\
Campus Beaulieu\\ F-35042  Rennes Cedex\\
 France}
\email{bachir.bekka@univ-rennes1.fr}

\address{Siegfried Echterhoff\\ Mathematisches Institut \\Universit\"at M\"unster \\
Einsteinstra\ss e 62, \\  D-48149 M\"unster\\
 Germany}
\email{echters@uni-muenster.de}
\author{Bachir Bekka}
\author{Siegfried Echterhoff}

\thanks{The first author acknowledges the support  by the ANR (French Agence Nationale de la Recherche)
through the projects Labex Lebesgue (ANR-11-LABX-0020-01) and GAMME (ANR-14-CE25-0004). 
This research was also funded by the Deutsche Forschungsgemeinschaft (DFG, German Research Foundation) under Germany's Excellence Strategy EXC 2044--390685587, Mathematics M\"unster: Dynamics--Geometry--Structure.}
\begin{abstract}
Let $\GG$ be an algebraic group over a local field $\kk$ of characteristic zero.
We show that the locally compact group $\GG(\kk)$ consisting of the $\kk$-rational points of  $\GG$ is of type I. Moreover, we complete Lipsman's characterization
of  the groups $\GG$  for which every irreducible
unitary representation  of $\GG(\kk)$ is a  CCR representation and
 show at the same time that such  groups $\GG(\kk)$ are trace class 
in the sense of \cite{Deitmar-vanDijk}.
\end{abstract}

\maketitle
\section{Introduction}
\label{S0}

Given a second countable locally compact group $G,$ a problem of major interest is to determine the \textbf{unitary dual space} $\widehat{G}$ of $G,$ that is, the set of equivalence classes of irreducible unitary representations  of $G$. 
The space $\widehat{G}$ carries a natural Borel structure, called the Mackey Borel structure, which is defined as follows (see \cite[\S 6]{Mackey-Borel} or \cite[\S 18.6]{Dixm--C*}). For every $n\in \{1, 2, \dots,\infty\},$
let $\Irr_n(G)$ denote the space of all irreducible unitary representations of $G$ in a fixed Hilbert
space $\H_n$ of dimension $n.$ The set $\Irr_n (G)$ is equipped with the weakest Borel structure  for which the functions $\pi\mapsto\langle \pi(g) \xi, \eta\rangle$
are measurable  for all $\xi, \eta\in \H_n$; the  disjoint union 
$\Irr(G)= \dot\cup_{n}\Irr_n (G)$ is endowed with the sum  Borel structure of the 
Borel structures of the $\Irr_n (G)$'s and $\widehat{G}$ with the quotient structure for the canonical surjective map $\Irr(G)\to\widehat{G}$.

A classification of  $\widehat{G}$ is only possible if  $\widehat{G}$ 
  is countably separated, that is, if there exists
a sequence  of Borel subsets of  $\widehat{G}$ which separates the points of 
$\widehat{G}.$  By a result of Glimm (\cite{Glimm}),  a group $G$ as above has this last property  if and only if $G$ is of type I in the sense of the following definition.

Recall that {a} von Neumann algebra  is a selfadjoint subalgebra of $\L(\H)$ which is closed for the weak operator topology of $\L(\H)$, where $\H$  is a  Hilbert space.
 A von Neumann algebra $\M$ is a factor if  the center of $\M$  consists  of the scalar operators. A unitary representation $\pi$ of $G$ in a Hilbert space $\H$ is a factor representation if the von Neumann subalgebra $W^*_\pi$  of $\L(\H)$ generated by $\pi(G)$ in $\L(\H)$ is a factor.

\begin{definition}
\label{TypeIGroup}
A locally compact group $G$ is of \textbf{type I} if, for every  factor representation
$\pi$ of $G$, the factor $W^*_\pi$ is of type I, that is, $W^*_\pi$ is isomorphic
to the von Neumann {algebra} $\L(\K)$ for some Hilbert space $\K;$ equivalently, if
$\pi$ is equivalent to a multiple $n\sigma$ of an irreducible unitary representation 
$\sigma$ of $G$ for some $n\in  \{1, 2, \dots,\infty\}.$
\end{definition}

The main result of   this note is to establish  the type I property for algebraic groups. 
Let $\KK$ be an algebraically closed field. A   linear algebraic group over
$\KK$ is a  Zariski closed subgroup of $GL_n(\KK)$  for some $n\geq 1.$
If $\kk$ is a subfield of $\KK$ and if $\GG$ is defined over $\kk$,
the group  {of} $\kk$-rational points of  $\GG$ is the subgroup 
$\GG(\kk)= \GG\cap GL_n(\kk).$
By {a} local field, we mean a non discrete locally compact field.
\begin{theorem}
\label{MainTheorem}
Let $\kk$ be a local field of characteristic $0$ and let $\GG$ be a linear algebraic group defined over $\kk$. 
The locally compact group $\GG(\kk)$ is of type I.
\end{theorem}
Some comments on Theorem~\ref{MainTheorem}\  
are in order.  
\begin{remark}
\label{Rem-MainTheorem}
\begin{enumerate}
\item[(i)] The case where $\GG$ is reductive   is due to Harish-Chandra (\cite{HC}) 
for $\kk=\RRR$ and to Bernstein (\cite{Bernstein}) for non archimedean $\kk$. 
\item[(ii)] In the case $\kk=\RRR$,  Theorem~\ref{MainTheorem} is a result  due to  Dixmier (\cite{Dixmier}).
\item[(iii)] Duflo (\cite{Duflo}) describes a general procedure to determine the unitary 
dual for a group $\GG(\kk)$ as in Theorem~\ref{MainTheorem}.
As is known to experts, one could  derive the fact that $\GG(\kk)$ is type I  from  Duflo's results, provided  Bernstein's result is extended to certain  finite central extensions of reductive groups.
This note offers  a direct approach, independent of \cite{Duflo}.
Indeed, our proof of Theorem~\ref{MainTheorem} closely follows the main strategy 
 in the case $\kk=\RRR$ from \cite{Dixmier} (see also \cite{Pukanszky-book}).
 In particular, we obtain a new proof of the known fact (\cite{Dixmier-Nilpotent}, \cite{Kirillov}, \cite{Moore}) that $\GG(\kk)$ is of type I when $\GG$ is unipotent.
 As a necessary preliminary step,  we  establish  the extension mentioned
 above of Bernstein's result   to  certain  finite central extensions of reductive groups
 (see Subsection~\ref{SSect:Covering}).
  \item[(iv)] Assume that $\kk$ is a local field of positive characteristic. 
 Then $\GG(\kk)$ is still of type I when $\GG$ is reductive  (\cite{Bernstein}).
To extend the result to an arbitrary  $\GG,$ the  major difficulty to overcome is to show that $\GG(\kk)$ is of type I when  $\GG$ is unipotent, a fact {which -- to our knowledge -- is} unknown.
\end{enumerate}
\end{remark}
Let $G$ be a  locally compact group.
Recall that, if $\pi$ is a unitary representation of $G$ on a Hilbert space $\H,$  then $\pi$ extends to a representation by bounded operators on $\H$ of the convolution algebra $C_c(G)$ of continuous functions with compact support on $G,$ 
defined by 
$$
\pi(f)=\int_G f(x) \pi(x) dx \tout f\in C_c(G),
$$
where $dx$ denotes a Haar measure on $G.$ 
The representation $\pi$ is said to be a  CCR representation if $\pi(f)$ is  a
compact operator for every $f\in C_c(G).$ The group $G$ is a 
\textbf{CCR group} if every $\pi\in \widehat{G}$ is a CCR representation.
It follows from  \cite[Theorem 9.1]{Dixm--C*} that every CCR group is type I.

Let $C_c^\infty(G)$ be the space of test functions on $G$ as defined 
in \cite{Bruhat}. Following \cite{Deitmar-vanDijk}, we
say that the representation $\pi$ of $G$ is of trace class if $\pi(f)$ 
is a trace class operator  for every $f\in C_c^\infty(G).$
The  group $G$ is called a \textbf{trace class}  group, if every irreducible unitary representation  of $G$ is trace class. Since $C_c^\infty(G)$ is dense in 
$L^1(G,dx),$ it is clear that a trace class group is CCR, and hence
of type I.

Let $\GG$ and $\kk$ be as in Theorem~\ref{MainTheorem}.
It is known that $\GG(\kk)$ is  a CCR group, if 
$\GG$ is reductive (see \cite{HC} and \cite{Bernstein})
or if $\GG$ is unipotent (see \cite{Dixmier-Nilpotent}, \cite{Kirillov}, \cite{Moore}).
This result has been strengthened in \cite{Deitmar-vanDijk} by showing 
 that $\GG(\kk)$ is a trace class group in these cases. 

Let $\UU$ be the unipotent radical of $\GG$ and $\LL$ a Levi subgroup defined
over $\kk;$ so, we have  a semi-direct product decomposition $\GG=\LL\UU.$
Let 
$$
\MM=\{g\in \LL\mid gu=ug \quad \text{for all}\quad u\in U\}.
$$
Then $\MM$ is an algebraic normal subgroup of $\LL$ defined over $\kk.$
Lipsman (\cite{Lipsman-CCR}) showed that, if $\GG(\kk)$ is CCR, then 
$\LL(\kk)/\MM(\kk)$ is compact. 
Our second  result concerns the converse statement
(see also the comment immediately after Theorem 3.1 in \cite{Lipsman-CCR}).
\begin{theorem}
\label{SecondTheorem}
Let  $\GG$ and $\kk$ be as in Theorem~\ref{MainTheorem}.
Let $\LL$ and $\MM$ be as above. The following properties are equivalent:
\begin{itemize}
\item[(i)] $\GG(\kk)$ is a CCR group;
\item[(ii)]   $\LL(\kk)/\MM(\kk)$ is compact;
\item[(iii)] $\GG(\kk)$ is trace class.
\end{itemize}
\end{theorem}
Theorem~\ref{SecondTheorem}\, generalizes
several results from \cite{vanDijk}  and provides a positive solution to a much stronger  version of  Conjecture 14.1 stated there.
\section{Some preliminary results}.
\label{S1}
\subsection{Projective representations}
\label{SS:ProjectiveRep}
 Let $G$ be a  second countable locally compact group.
We will  need the notion of a projective  representation of $G$.
For all what follows, we refer to  \cite{Mackey}.

Let $\H$ be a Hilbert space.  Recall that a  map $\pi: G \to U(\H)$
from $G$ to the unitary group   of  $\H$
is  a \textbf{projective representation} of $G$
if the following holds:
\begin{itemize}
 \item[(i)] $\pi(e)=I$,
\item[(ii)] for all $x,y\in G,$ there exists  $\omega(x , y )\in {\mathbf T} $
such that 
$$\pi(x y ) = \omega(x , y )\pi(x )\pi(y ),$$
\item[(iii)] the map  $g\mapsto \langle\pi(g)\xi,\eta\rangle$ is {Borel} for all
$\xi,\eta\in\H.$
\end{itemize}
The map  $\omega:G \times G \to {\mathbf T}$ has  the following properties: 
\begin{itemize}
 \item[(iv)]
$\omega(x,e)=\omega(e, x)$ for all $x\in G,$
\item[(v)]
 $\omega(xy, z)\omega(x, y)= \omega(x, yz)\omega(y, z)$ for all $x,y, z\in G.$
\end{itemize}
The set  $Z^2({G},  {\mathbf T})$  of all 
{Borel} maps
  $\omega:G \times G \to {\mathbf T}$ with properties (iv) and (v)  is 
an abelian group for the pointwise product.

 For a given $\omega\in Z^2(G, \mathbf T),$ a  map $\pi: G \to U(\H)$
 with properties (i), (ii) and (iii) as above is called an \textbf{$\omega$-representation}
 of $G.$

For $i=1,2,$ let  $\omega_i, \in Z^2(G, \mathbf T)$
and  $\pi_i$ an $\omega_i$-representation of $G$ on a Hilbert space $\H_i.$
Then $\pi_1\otimes \pi_2$ is an $\omega_1 \omega_2$-representation of $G$ on the Hilbert space $\H_1
 \otimes \H_2.$ In particular, if $\omega_2= \omega_1^{-1},$
then  $\pi_1\otimes \pi_2$  is an ordinary representation of $G.$

Every projective unitary representation of $G$ can be  lifted to 
an ordinary unitary representation of a central extension
of $G.$ More precisely, 
for $\omega\in Z^2(G, \mathbf T),$ let 
$S$ be the closure in $\mathbf T $ of the subgroup 
generated by  the image of $\omega$; 
define a group $G^{\omega}$ with underlying set $S \times G$
and multiplication $(s,x) (t,y) \, = \, (st\omega(x,y), xy)$.
Equipped with a suitable topology, $G^{\omega}$ is a locally compact group.
Let $\pi: G \to U(\H)$ be an $\omega$-representation of $G.$
Then $\pi^0:G^{\omega} \to U(\H),$ defined by $\pi^0(s,x)= s\pi(x)$
is an ordinary representation of $G^{\omega}$; moreover, the map
$\pi\to \pi^0$ is a bijection between the  $\omega$-representations of $G$
and the representations of $G^{\omega}$ for which the restriction of
the central subgroup $S\times\{e\}$ is a multiple of the one dimensional
representation $(s,e)\mapsto s.$

A projective representation $\pi$  of $G$ on $\H$ is irreducible
(respectively, a factor representation) if the only closed $\pi(G)$-invariant subspaces of $\H$ are $\{0\}$ and $\H$  (respectively, if the von 
Neumann algebra generated by $\pi(G)$ is a factor).
A factor projective representation $\pi$ is of type $I$,
if the factor  generated by $\pi(G)$ is of type $I.$

\subsection{Factor representations}
\label{SS:FactorRep}
 Let $G$ be a  second countable locally compact group.
Let $N$ be a {type I} closed normal subgroup of $G$.
Then $G$ acts by conjugation  on $\widehat{N}:$
for $\pi\in \widehat{N},$ the conjugate representation $\pi^g\in \widehat{N}$ 
is defined by $\pi^g(n)= \pi(gng^{-1}),$ for $g\in G, n\in N.$

Let $\pi\in \widehat{N}.$ The stabilizer 
$$
{ G}_\pi=\{g\in G\ :\ \pi^g\ \text{is equivalent to}\ \pi\}
$$ 
of $\pi$ is a closed subgroup of $G$ containing $N.$
 There exists a projective unitary representation $\widetilde\pi$
of ${G}_{\pi}$ on $\H$ which extends $\pi$. 
Indeed, for every
$g\in {G}_\pi$,  there exists a unitary
operator $\widetilde\pi(g)$ on $\H$ such that
$$
\pi^g(n)= \widetilde\pi(g) \pi(n) \widetilde\pi(g)^{-1} \tout n\in N.
$$
One can choose  $\widetilde\pi(g)$  
such that $g\mapsto \widetilde\pi(g)$ is an $\omega\circ (p\times p)$-{representation
 of ${ G}_\pi$} which extends $\pi$
for some $\omega \in Z^2(G/N, \mathbf{T}),$
where $p:G\to G/N$ is the canonical projection (see Theorem~8.2 in \cite{Mackey}).

 The normal subgroup $N$ is said to be \textbf{regularly embedded}\ in $G$ if
the quotient space $\widehat{N}/G$, equipped with the quotient Borel structure
inherited from $\widehat{N}$, is  a countably separated Borel space.
 
 The following result  is part of the so-called Mackey machine;
it is a basic tool which reduces  the determination of the factor
(or irreducible) representations of  $G$ to  the  determination of the factor
(or irreducible) representations of  subgroups of $G$.
The proof is a direct consequence 
of Theorems 8.1, 8.4 and 9.1 in \cite{Mackey}.

\begin{theorem}
\label{Mackey}
Let $N$ be a  closed normal subgroup of $G$ and denote by 
$p:G\to G/N$ the canonical projection.  Assume that $N$ is 
of type I and is regularly embedded in $G.$
\begin{itemize}
\item[(i)] Let $\pi\in \widehat{N}$ and $\pi'$
a factor   representation of ${G}_{\pi}$ such that the restriction of 
$\pi'$ to $N$ is a multiple of $\pi.$ 
Then the induced representation $\Ind_{G_\pi}^G \pi'$
is a factor representation. Moreover, $\Ind_{G_\pi}^G \pi'$ is of type I 
if and only if  $\pi'$ is of type I.
\item[(ii)]  Every  factor representation  of $G$ is equivalent to a 
representation  of the form $\Ind_{G_\pi}^G \pi'$ as in (i).
\item[(iii)] Let $\pi\in\widehat{N}$ and assume that $G=G_\pi.$
Let $\omega \in Z^2(G/N, \mathbf{T})$ and $\widetilde\pi$  {an}  $\omega\circ (p\times p)$-representation of ${ G}$ which extends $\pi$.
Then every   factor   representation $\pi'$ of ${G}$ such that the restriction of 
$\pi'$ to $N$ is a multiple of $\pi$ is equivalent to a representation of the form
$\sigma \otimes \widetilde\pi,$ where $\sigma$ is a factor  $\omega^{-1}$-representation
of $G/N$ lifted to $G.$ Moreover,  $ \pi'$ is of type I if and only if 
$\sigma$ is of type I.

\end{itemize}
\end{theorem}
We will need the well-known fact that being of type I is inherited from cocompact normal subgroups; 
 for the sake of completeness, we reproduce
a  short proof modeled after  \cite{Pukanszky-book}.

\begin{proposition}
\label{Prop-FiniteIndex}
Let $N$ be a normal subgroup  of $G$.
Assume that $N$ is of type I and that $G/N$ is compact. Then $G$ is  of type I.
\end{proposition}

\begin{proof}
Since $N$ is of type I,  the Borel space $\widehat{N}$ is countably separated,
by Glimm's result (mentioned in the introduction). The action of $G$ on $\widehat{N}$ factorizes to an action of the compact  group $K:= G/N.$

We claim that $\widehat{N}/K$ is countably separated, that is, $N$ is regularly 
embedded in $G.$ Indeed, 
by a theorem of Varadarajan (see \cite[2.1.19]{Zimmer}), there exists a compact 
metric space $X$ on which $K$ acts by homeomorphisms and an injective
$K$-equivariant Borel map $\widehat{N}\to X.$ Since $K$ is compact, $X/K$ is easily seen to be
 countably separated and the claim follows.

Let  $\pi\in \widehat{N}$ and {$\sigma$ a projective} factor representation 
of $G_\pi/N$. Then $\sigma$  lifts to an ordinary representation $\widetilde{\sigma}$
of a central extension $\widetilde{G_\pi}$ of $G_\pi/N$ by a closed subgroup of $\mathbf{T}.$ Since $G_\pi/N$ is compact, $\widetilde{G_\pi}$ is compact and therefore
of type I. So, $\widetilde{\sigma}$ and hence $\sigma$ is of type I.
Theorem~\ref{Mackey} shows that  $G$ is of type I.

\end{proof}

We will use the following  consequence of  Proposition~\ref{Prop-FiniteIndex} in the proof of Theorem~\ref{SecondTheorem}.
\begin{corollary}
\label{lem-subrep}
Let $G$ and $N$ be as in Proposition~\ref{Prop-FiniteIndex}.
For every $\rho\in \widehat{G},$ there exists a representation $\pi\in \widehat{N}$ such that $\rho$ is a subrepresentation  of $\ind_N^G\pi$.
\end{corollary}
\begin{proof} 
Let $\rho\in \widehat{G}$. By Proposition~\ref{Prop-FiniteIndex}, $G$ is of type I.
 Hence, by Theorem \ref{Mackey}, there exist $\pi\in \widehat{N}$
and  $\sigma\in \widehat{G}_\pi$ such that $\sigma|_N$ is  a multiple  of $\pi$ 
and such that $\rho \simeq \Ind_{G_\pi}^G\sigma$. 

As is well-known (see for instance \cite[Proposition E.2.5]{BHV}),
 $$\Ind_N^{G_\pi}(\sigma|_N) \simeq \sigma \otimes \Ind_N^{G_\pi}1_N,$$ 
that is, 
$$ \Ind_N^{G_\pi}(\sigma|_N)\simeq  \sigma \otimes  \lambda_{G_\pi/N},$$ where 
$\lambda_{G_\pi/N}$ is the regular representation of $G_\pi/N$ lifted
to $G_\pi.$ Since $G_\pi/N$ is compact, $\lambda_{G_\pi/N}$ contains
the trivial representation $1_{G_\pi}.$  It follows that  $\Ind_N^{G_\pi}(\sigma|_N)$
contains $\sigma$.
Since  $\sigma|_N$ is  multiple  of $\pi$ and since $\sigma$ is irreducible,
we see that $\sigma$ is contained $\Ind_N^{G_\pi}\pi$.
Induction by stages shows then that $\rho\simeq \Ind_{G_\pi}^G\sigma $
is contained in  $ \Ind_{G_\pi}^G(\Ind_N^{G_\pi}\pi)\simeq \Ind_N^G\pi.$

\end{proof}

\subsection{Actions of algebraic groups}
\label{SS:Actions}
Assume now that $\kk$ is a local field of characteristic $0$ and $\GG$ a linear algebraic group defined over $\kk$. 
If $V$ is an algebraic variety defined {over $\kk$,} then the set $V(\kk)$ of $\kk$-rational points in
$V$ has also a locally compact topology, {which} we call the Hausdorff topology. 

The following well-known result (see \cite[3.1.3]{Zimmer}) is a crucial tool
for  the sequel. We indicate briefly the main steps in its proof.
\begin{theorem}
\label{Theo-ActAlg}
 Let
$\GG\times  V\to V$ be a $\kk$-rational action of $\GG$ on an algebraic  variety $V$ defined over $\kk$ .
Then every $\GG(\kk)$-orbit in $V(\kk)$ is open in its closure for the Hausdorff topology.
\end{theorem}
\begin{proof}
Let $W$ be a $\GG$-orbit in $V.$ Then 
$W$ is open in its closure for the Zariski topology (see \cite[1.8]{Borel}).
This implies that $W(\kk)$ is open in its closure for the  Hausdorff topology.
By \cite[6.4]{Borel-Serre}, there are only finitely many $\GG(\kk)$-orbits
contained in $W(\kk)$ and the claim follows.
\end{proof}

We will use the previous theorem in the case of a linear representation of $\GG.$
More precisely, let $\KK$ be an algebraic closure of $\kk$
and let $V$ be   a finite dimensional vector space over $\KK$,
equipped with  a $\kk$-structure $V_{\kk}$. Let 
$\rho:\GG\to GL(V)$ be a $\kk$-rational representation of $\GG$.
Consider the dual  adjoint  $\rho^*$ of $\GG$ on the dual vector space 
$V^*={\rm Hom}(V, \KK)$, equipped with  the $\kk$-structure $V_{\kk}^*= {\rm Hom}(V_{\kk}, \kk)$. Then $\rho^*:\GG\to GL(V^*)$  is a $\kk$-rational representation of $\GG$.
Fix a non-trivial unitary character $\eps\in \widehat{\kk}$ of the additive group of $\kk.$
The map $\Phi:V_{\kk}^*\to \widehat{V_{\kk}},$ given by 
$$
\Phi(f)(v)=\eps\left({f(v)}\right) \tout f\in V^* \quad\text{and} \quad v\in V_\kk,
$$ 
is an isomorphism of topological groups (see \cite[Chap.II, \S5]{Weil}).
The group $\GG(\kk)$ acts on  $V_{\kk}, V_{\kk}^*$ and  $\widehat{V_{\kk}}$;
the map $\Phi$ is $\GG(\kk)$-equivariant.
Therefore, the following corollary is a direct consequence of Theorem~\ref{Theo-ActAlg}.
\begin{corollary}
\label{Cor-DualRep}
Let $\rho:\GG\to GL(V)$ be a $\kk$-rational representation of $\GG$ as above.
Every $\GG(\kk)$-orbit in $\widehat{V_{\kk}}$ is open in its closure.
\end{corollary}

Let $\UU$ be the unipotent radical of $\GG$ and $ \mathfrak{u}$ its Lie algebra.
Then $\UU$ is an algebraic subgroup of $\GG$ defined over $\kk$  and  the exponential map
$\exp: \mathfrak{u}\to \UU$ is an isomorphism of $\kk$-varieties;
moreover, $\mathfrak{m}\to \exp \mathfrak{m}$ is a bijection between
Lie subalgebras of $\mathfrak{u}$ and (connected) algebraic subgroups of $\UU.$

The Lie algebra    $\mathfrak{u}$  has  a $\kk$-structure ${\mathfrak{u}}_{\kk}$
which is  the Lie $\kk$-subalgebra of $\mathfrak{u}$ 
for which $\exp: {\mathfrak{u}}_{\kk}\to \UU(\kk)$ is a bijection.

The action of  $\GG$ by conjugation {on $\UU$  induces} an action, called
the \textbf{adjoint representation} $\Ad:\GG\to GL(\mathfrak{u})$,
which is  a $\kk$-rational representation by automorphisms  of the Lie algebra $\mathfrak{u}$ for which
$\exp: \mathfrak{u}\to \UU$ is $\GG$-equivariant.
The map $\mathfrak{m}\to \exp \mathfrak{m}$ is a bijection between
$\Ad(\GG)$-invariant ideals of  $\mathfrak{u}$ and algebraic normal subgroup of $\GG$ contained in $\UU.$

\begin{lemma}
\label{Lem-RegEmbedded}
Let $\MM$ be an abelian algebraic normal subgroup of $\GG$ contained in $\UU.$
Then $\MM(\kk)$ is regularly embedded in $\GG(\kk).$
\end{lemma}
\begin{proof}
By an elementary argument (see \cite[2.1.12]{Zimmer}), it suffices to show that every $\GG(\kk)$-orbit in $\widehat{\MM(\kk)}$ is open in its closure. 
Let  $\mathfrak{m}$ be the $\Ad(\GG)$-invariant  ideal of  $\mathfrak{u}$ corresponding
to $\MM$ and set $\mathfrak{m}_{\kk}= \mathfrak{m}\cap \mathfrak{u}_\kk.$ 
Since $\MM$ is abelian, $\exp: \mathfrak{m}_{\kk}\to \MM(\kk)$ is an isomorphism of 
topological groups. So, it suffices to show that every $\GG(\kk)$-orbit in $\widehat{\mathfrak{m}_{\kk}}$ is open in its closure.
This is indeed the case by Corollary~\ref{Cor-DualRep}.

\end{proof}

\subsection{Heisenberg groups and Weil representation}
\label{SS:Heisenberg}
Let $\kk$ be a   field of characteristic $0$ and let $n\geq 1$ be an integer. The Heisenberg group $H_{2n+1}(\kk)$
 is the  nilpotent group with underlying set $\kk^{2n}\times \kk$ and product
$$((x,y),s)((x',y'),t)=\left((x+x',y+y'),s + t + \dfrac{1}{2}\beta((x,y),(x',y'))\right),
$$
for $(x,y), (x',y')\in \kk^{2n},\, s,t\in \kk,$
where $\beta$ is the standard symplectic form on $\kk^{2n}$.
The group $H_{2n+1}(\kk)$ is the group of $\kk$-rational points of a unipotent algebraic
 group  $H_{2n+1}$ defined over $\kk.$
 Its Lie algebra $\mathfrak{h}_{2n+1}$ has a basis 
 $\{X_1, \dots, X_n, Y_1, \dots, Y_n, Z\}$ with non trivial commutators
 $[X_i, Y_i]=Z$ for all $i=1,\dots, n.$
The center of $\mathfrak{h}_{2n+1}$ is equal to the commutator subalgebra 
$[\mathfrak{h}_{2n+1}, \mathfrak{h}_{2n+1}]$ and is  spanned by $Z$.
 This last property characterizes $H_{2n+1}$; for the proof, see \cite[Lemme 4]{Dixmier}.
 \begin{lemma}
 \label{Lem-Heisenberg}
 Let $\UU$ be a unipotent algebraic group defined over $\kk$ with Lie algebra $\mathfrak{u}$.
 Assume that the center $\mathfrak{z}$ of $\mathfrak u$ equals $[\mathfrak{u},\mathfrak{u}]$ and that 
 $\dim \mathfrak{z}=\dim ([\mathfrak{u},\mathfrak{u}])=1$.
 Then $\UU$ is isomorphic to $H_{2n+1}$ over $\kk$ for some $n\geq 1.$
  \end{lemma}
 
The symplectic group $Sp_{2n},$ which is the isometry group of $\beta,$ 
acts by rational automorphisms of $H_{2n+1},$ given by 
$$
 \vfi_g((x,y),t)= (g(x,y),t) \tout g\in Sp_{2n},\ ((x,y),t) \in H_{2n+1}.
$$
Let $\Aut_c(H_{2n+1})$ be the group of automorphisms of $H_{2n+1}$
which acts trivially on the center of $H_{2n+1}$. Observe that $\vfi_g\in\Aut_c(H_{2n+1})$
for every $g\in Sp_{2n}$ and that $I_h\in\Aut_c(H_{2n+1})$ for every $h\in H_{2n+1},$ 
where $I_h$ is the inner automorphism given by $h.$
The following proposition is proved in \cite[1.22]{Folland} in the case $\kk=\RRR;$
however, its proof is valid in our setting.
\begin{proposition}
\label{Pro-AutHeisenberg}
Every automorphism in $\Aut_c(H_{2n+1})$ 
can be uniquely written as the product $\vfi_g\circ I_h$ 
for $g\in Sp_{2n}$ and $h\in H_{2n+1}.$
So, $\Aut_c(H_{2n+1})$  can be identified with
the semi-direct product $Sp_{2n}\ltimes H_{2n+1}.$ 
\end{proposition}
The following facts about the irreducible representations 
of $H_{2n+1}(\kk)$  and associated metaplectic (or oscillator)
representations of $Sp_{2n}(\kk)$  will play a crucial role in the sequel;
for more details, see \cite{Weil}.
\begin{theorem}
\label{Theo-StoneVN-Weil}
Denote by $Z$ the center of $H_{2n+1}(\kk)$
and let   $\chi \in\widehat{Z}$ be a non trivial character of $Z.$
\begin{itemize}
\item[(i)] \textbf{(Stone-von Neumann)} 
 There exists, up to equivalence, a \emph{unique}
irreducible unitary representation $\pi_\chi$  of $H_{2n+1}(\kk)$
such that $\pi_\chi|_{Z}$ is a multiple of $\chi.$
\item[(ii)] \textbf{(Segal-Shale-Weil)} 
 The representation $\pi_\chi$ extends to a unitary 
 representation of 
 $\widetilde{Sp_{2n}(\kk)}\ltimes H_{2n+1}(\kk),$
 where $\widetilde{Sp_{2n}(\kk)}$ is  a twofold  cover of $Sp_{2n}(\kk),$
 called the \emph{metaplectic} group.
\end{itemize}

\end{theorem}
\subsection{Large compact subgroups and admissible representations}
Let $G$ be a locally compact group. 
We introduce a few notions from \cite[7.5]{Warner}. 
\begin{definition}
\label{Def-LargeCompact}
Let $K$ be a compact subgroup of $G.$
\begin{itemize}
\item[(i)] $K$ is said to be \textbf{large} in $G$ if, for every $\sigma\in \widehat{K}$,
$$\sup_{\pi \in \widehat{G}} m(\sigma, \pi|_K)<+\infty,$$
where $m(\sigma, \pi|_K)$ is the multiplicity    of $\sigma$
in $\pi|_K.$
\item[(ii)]
$K$ is said to be  \textbf{uniformly large} in $G,$ if there exists an integer 
$M$ such that, for every $\sigma\in \widehat{K}$, 
$$\sup_{\pi \in \widehat{G}} m(\sigma, \pi|_K)\leq M \dim \sigma,$$
\end{itemize}
\end{definition}
\begin{proposition}
\label{Pro-LargeCompactSubgroups}
Let $G$ be a second countable  locally compact group.
\begin{itemize}
\item[(i)] Assume that $G$ contains a large compact subgroup. Then $G$ is CCR.
\item[(ii)] Assume that $G$ contains a uniformly large compact subgroup.
Then $G$ is trace class.
\end{itemize}
\end{proposition}
\begin{proof}
Item (i) is well known (see \cite[Theorem 4.5.7.1]{Warner}).

Assume that $G$ contains a uniformly large compact subgroup.
Let $f\in C_c(G)$ and $\pi\in \widehat{G}.$ Then
$\pi(f)$ is a Hilbert-Schmidt operator (\cite[Theorem 4.5.7.4]{Warner}).
It follows from \cite[Proposition 1.6]{Deitmar-vanDijk} that 
$\pi(f)$ is a trace class operator if $f\in C_c^\infty(G)$.
\end{proof}

We introduce further notions in the context of totally disconnected groups.
\begin{definition}
\label{Def-AdmissibleComp}

\begin{itemize}
\item[(i)] A representation $\pi: G\to GL(V)$  in a complex vector space $V$ 
is said to be \textbf{smooth}
if the stabilizer in $G$ of every $v\in V$ is open. 
\item[(ii)]  A smooth representation or unitary representation $(\pi, V)$  
  of $G$   is \textbf{admissible}
if  the space $V^K$ of $K$-fixed vectors in $V$   is finite dimensional, for
every  compact open subgroup $K$ of   $G.$
\item[(ii)]  $G$ is said to have a \textbf{uniformly admissible} smooth dual 
 (respectively,  unitary dual) if, for every compact open subgroup 
 $K$ of $G,$ there exists a constant $N(K)$ such that
$\dim V^K\leq N(K),$ for every irreducible smooth (respectively,
irreducible unitary) representation $(\pi, V)$   of $G$.
\end{itemize}
\end{definition}
Let  $G$ be a   totally disconnected  locally compact group.
For a  compact open subgroup $K$ of   $G,$ let 
$\H(G, K)$ be the convolution algebra of continuous functions
on $G$ which are bi-invariant under $K.$ 
The algebra $\H(G,K)$ is a $*$-algebra,
for the involution given by 
$$
f^*(g)=\Delta(g)\overline{f(g^{-1})} \tout f\in\H(G,K), g\in G,
$$
where $\Delta$ is the  modular function of $G.$
Observe that 
$$\H(G,K)= e_K \ast C_c(G) \ast e_K,$$ where
$e_K= \dfrac{1}{\mu(K)}\Un_K$ and $\mu$ is a Haar measure on $G.$

Let $(\pi, V)$ be a  smooth (respectively, unitary) representation of $G$
and let $K$ be a compact open subgroup of $G.$
A representation (respectively, a $*$-representation) $\pi_K$ of  $\H(G,K)$ is
defined on $V^K$ by 
$$
\pi_K(f) v:=\int_{G}f(g) \pi(g)v d\mu(g)
\tout f\in \H(G,K),\,v\in V^K.
$$
If $\pi$ is  irreducible, then  $\pi_K$ is  an algebraically (respectively, 
topologically) irreducible  representation of $\H(G,K).$
Moreover, every  algebraically (respectively, 
topologically) irreducible  representation of $\H(G,K)$ is of the form 
$\pi_K$ for some irreducible smooth (respectively, unitary) representation $\pi$ of $G$.

\begin{proposition}
\label{Pro-AdmissibleCompactSubgroups}
Let $G$ be a totally disconnected  locally compact group.
\begin{itemize}
\item[(i)] Assume that every   irreducible unitary representation of $G$     is  admissible. Then $G$ is  trace class.
\item[(ii)] Assume that $G$ has a uniformly admissible smooth dual.
Then $G$ has  a uniformly admissible unitary  dual.
 \item[(iii)] Assume that $G$ has  a uniformly admissible unitary dual.
 Then every  compact open subgroup of   $G$ is large.

\end{itemize}
\end{proposition}
\begin{proof}

To show item (i), observe that the algebra $C_c^\infty(G)$ of test functions on $G$ 
mentioned in the introduction is the union $\bigcup_{K} \H(G,K),$ where
$K$ runs over the compact open subgroups  of   $G.$
Let $K$ be such a subgroup and $(\pi, V)$  an irreducible unitary representation of $G$. Since $\pi(e_K)$  is the orthogonal projection 
on $V^K$ and since  $V^K$ is finite dimensional, $\pi(f)= \pi(e_K)\pi(f) \pi(e_K)$  has finite rank, for every $f\in \H(G,K)$; see also \cite[Theorem 2.3]{Deitmar-vanDijk}.

Item (ii) is proved in \cite[Theorem B]{First-Rud}.

To show  item (iii), let $K$ be a
 compact open subgroup of $G$ and let $\sigma\in  \widehat{K}$.
  Let $(\pi, \H)$ be an irreducible representation
of $G$. Since  $K$ is a totally disconnected compact group, there exists a normal open subgroup $L=L(\sigma)$
of $K$ such that $\sigma|_L$ is the identity representation.
By assumption, there exists an integer
$N(L)$ such that $\dim \H^L \leq N(L)$
for every $(\pi, \H) \in \widehat{G}.$ 
Since  
$$ m(\sigma, \pi|_K) \dim \sigma\leq \dim \H^L,
$$
the claim follows.
\end{proof}

The following immediate corollary is a useful criterion
for the trace class property in the totally disconnected case.

\begin{corollary}\label{cor-CCR-trace-class}
A totally disconnected locally compact group $G$ is trace class if and only if $G$ is CCR.
\end{corollary}
\begin{proof} 
As already mentioned, $G$ being trace class implies that $G$ is CCR.
Conversely, assume that $G$ is CCR.
Then, for every compact open subgroup $K$  and every irreducible 
representation $\pi$ of $G$, the projection $\pi(e_K)$ is a compact 
operator  and  has therefore finite rank. Hence, every irreducible unitary representation of $G$ is admissible  and $G$ is  trace class by Proposition~\ref{Pro-AdmissibleCompactSubgroups}.(i).
\end{proof}

As mentioned in the introduction, it is known that reductive algebraic groups
over local fields are  of type I. In fact, 
the following stronger result holds.
\begin{theorem}
\label{Theo-TypI-Red}
\begin{itemize}
\item[(i)] \textbf{(Harish-Chandra)} Let $G$ be a connected reductive Lie group with finite center.
Then every maximal compact subgroup  of $G$ is uniformly large in $G.$
\item[(ii)] \textbf{(Bernstein)} Let  $G=\GG(\kk)$, where $\GG$ is  a reductive linear algebraic group $\GG$ over 
a non archimedean local field $\kk$.  Then  $G$ has a uniformly admissible smooth dual
and hence a uniformly admissible unitary dual.
\end{itemize}
Every group  $G$  as in (i) or (ii) is trace class.
\end{theorem}
\begin{proof}
For (i), we refer  to \cite{HC}. 
Item (ii) follows from \cite{Bernstein}\ in combination with results from 
\cite{HC-vanDijk}. Alternatively, Item (ii) is a direct consequence
of \cite{Bernstein}, of  the fact that   $G$ has an admissible smooth dual
(\cite{BernsteinNotes}),  and of Proposition~\ref{Pro-AdmissibleCompactSubgroups} (ii).

The last statement follows from  Propositions~\ref{Pro-LargeCompactSubgroups} and \ref{Pro-AdmissibleCompactSubgroups}.
\end{proof}

\subsection{Covering groups of reductive groups}
\label{SSect:Covering}
Let $\kk$ be a local field  and $\GG$ a reductive linear algebraic group defined over $\kk$.
When $\kk$ is archimedean, every finite covering group of $G=\GG(\kk)$  is of type I  (see Theorem~\ref{Theo-TypI-Red} (i)).
We need to show that this  result also holds in the non archimedean case, 
 at least for certain finite  covers of $G$.

As the following example shows, a central finite extension
of a locally compact group of type I is in general not of
type I.

 \begin{example}
\label{HeisenbergFp}
 Let $\FF_p$ be the field of order $p$ for a  prime $p \ge 3$.
Let $V$ be the  vector space over $\FF_p$ of  sequences $(x_n)_{n\in \N}$ 
with $x_n\in \FF_p$ and $x_n = 0$ for almost all $i \in \N$,
and let  $\omega$ be the symplectic form on $V \oplus V$, defined by
$$
\omega((x,y),(x',y')) \, = \, \sum_{n \in \N} (x_ny'_n -y_nx'_n)
\hskip.5cm \text{for} \hskip.2cm 
(x,y), (x',y') \in V \oplus V.
$$
Let  $H_\infty(\FF_p)$ be the discrete group with underlying set $V \oplus V \oplus \FF_p$
and with multiplication defined by
$$
(x,y,z) (x',y',z') \, = \, (x+x', y+y', z+z'+ \omega((x,y),(x',y')))$$
for all $(x,y,z), (x',y',z') \in H_\infty(\FF_p).$

The ``Heisenberg group" $H_\infty(\FF_p)$ is a finite central cover of the abelian group $ V\oplus V.$
It is easily checked that $H_\infty(\FF_p)$ is not  virtually abelian.
So,  $H_\infty(\FF_p)$
is not of type I, by Thoma's characterization of discrete groups of type I
(\cite{Thoma}).
\end{example}
 
 Let  $\kk$ be a  non archimedean  local field,   $\GG$ a reductive linear algebraic group defined over $\kk$, and $G= \GG(\kk)$.

 Let $\widetilde{G}$ be a finite central extension of $G$,
that is, there exist a finite normal subgroup $F$ contained in the center of 
$\widetilde{G}$ and a continuous surjective homomorphism
$p: \widetilde{G} \to G$ with $\ker p=F.$

Let $P$ be a parabolic subgroup of $G$.
The subgroup $\widetilde P= p^{-1}(P)$ is called a \textbf{parabolic subgroup}
of $\widetilde{G}.$
Let $N$ be the unipotent radical of $P$
and  $P=MN$ a Levi decomposition.
 There exists a unique closed subgroup $\underline{N}$ of $\widetilde{G}$
 such that $p|_{\underline{N}}: \underline{N}\to N$ is an isomorphism (see \cite[11 Lemme]{Duflo}).
By its uniqueness property, $\underline{N}$ is a normal subgroup of $\widetilde{P},$
called the \textbf{unipotent radical}\, of $\widetilde P$, and we have a \textbf{Levi decomposition}  $\widetilde P= \widetilde M \underline{N},$ 
 where   $\widetilde M= p^{-1}(M).$

We recall a few facts about the structure of $G=\GG(\kk)$ from \cite{Bruhat-Tits}.

 Let $P_0$ be a fixed minimal parabolic subgroup of $G$
 with unipotent radical $N_0$
 and   Levi decomposition  $P_0=M_0N_0$, where $N_0$ is the unipotent radical
 of $P_0$.
Let  $A_0$ be a maximal split torus contained
 in $P_0.$ 
 For any root $\alpha$ of $A_0,$ let $\xi_\alpha$ be  the corresponding character of 
 $A_0$ and let $A_0^+$ be the set of all $a\in A_0$ with $|\xi_\alpha(a)|\geq 1.$
Let   $Z$ be a maximal split torus  contained in the center of $G.$
Let $P_0^-$ be the opposite parabolic subgroup to $P_0$ and 
$N_0^-$ the unipotent radical of $P_0^-.$

 There exists a   compact maximal subgroup $K_0$ of $G$ such that 
 $G=K_0 P_0$ (Iwasawa decomposition). 
 We have a Cartan decomposition 
 $$G= K_0 S Z \Omega K_0,$$ where $S$ is a finitely generated
  semigroup contained in $A_0^+$ 
   and $\Omega$ is a finite subset of $G;$
   moreover, the following holds:
 every neighbourhood  of $e$ contains a compact open subgroup 
$K$ with the following properties:
\begin{itemize}
\item[(i)] $K\subset K_0$ and $K_0$ normalizes $K;$
\item[(ii)] We have 
$K= K^+ K^-,$ where $K^+= K\cap P_0$ and $K^-= K\cap N_0^-;$ 
\item[(iii)] $a^{-1}K^+ a\subset K^+$ and $a K^- a^{-1}\subset K^-$ for every $a\in S.$
  \end{itemize}
For all this, see \cite{Bruhat-Tits}.

We assume from now on that $\widetilde{G}$ satisfies the following  Condition $(*)$:
\begin{itemize}
\item[$(*)$] For every $g,h\in \widetilde{G},$ we have
$$gh=hg\Longleftrightarrow p(g)p(h)=p(h)p(g).$$
\end{itemize}
Set 
$$
\widetilde{K_0}= p^{-1}(K_0),\quad \widetilde{S}= p^{-1}(S),\quad \widetilde{Z} = p^{-1}(Z),  \quad\text{and} \quad\widetilde{\Omega}= p^{-1}(\Omega).
$$
Let $\underline{N_0}$ and $\underline{N_0^-}$ be the unique subgroups of $\widetilde{G}$ 
corresponding to $N_0$ and $N_0^-$.

\begin{proposition}
\label{Pro-Covering}
We have 
\begin{enumerate}
\item $\widetilde G= \widetilde K_0 \widetilde S \widetilde Z \widetilde \Omega \widetilde K_0;$
\item $\widetilde Z$ is contained in the center of $G;$
\item $\widetilde S$ is a finitely generated commutative semigroup of $\widetilde{G};$
\item every neighbourhood  of $e$ in  $\widetilde{G}$ contains a compact open subgroup 
$L$ with the following properties:
\begin{itemize}
\item[(4.i)] $L\subset \widetilde{K_0}$ and $ \widetilde{K_0}$ normalizes $L;$
\item[(4.ii)] $L= L^+ L^-,$ where $L^+= L\cap \widetilde{P_0}$ and $L^-= 
L\cap \underline{N_0^-};$\
\item[(4.iii)] $\widetilde{a}^{-1}{L^+} \widetilde{a}\subset L^+$ and $\widetilde{a} L^-\widetilde{a}^{-1}\subset L^-$ for every $\widetilde{a}\in \widetilde S .$
\end{itemize}
\end{enumerate}
 
\end{proposition}
\begin{proof}
It is clear that property (1) holds.
It follows from Condition $(*)$ that 
$ \widetilde{Z}$ is contained in the center of  $\widetilde{G}$
and that $ \widetilde{S}$ is commutative. 
Moreover, $ \widetilde{S}$ is a  finitely generated commutative semigroup of $\widetilde{G},$ since $\widetilde G$ is a finite 
covering of $G.$
So,  properties (2) and (3) are satisfied.

Let $\widetilde U$ be an open neighbourhood of $e$ in 
$\widetilde{G}$ and set $U:= p(\widetilde U).$
We may assume that $\widetilde{U} \cap F=\{e\},$
where $F= \ker p$; so,
 $$p|_{\widetilde U}: \widetilde{U} \to U$$ 
is a homeomorphism and we can therefore  identify 
$p^{-1}(U)$
with $U\times F$ and  $\widetilde U$ with the open open subset $U\times \{e\}$
of $U\times F$.

Let $ \widetilde{a_1}, \dots,  \widetilde{a_n}$ be a generating set 
of $\widetilde{S}.$
Choose a neighborhood $V$ of $e$ contained in $U$ 
such that 
$V^{-1}= V,$ $V^2\subset U,$ and 
$$\widetilde{a_i} (V\times \{e\}) \widetilde{a_i}^{-1} \subset U\times \{e\}
\quad\text{and} \quad \widetilde{a_i}^{-1} (V\times \{e\})  \widetilde{a_i}  \subset U\times \{e\}.
$$
for every $i\in\{1,\dots, n\}$.

Fix a  compact open subgroup $K_1$
contained in $V$.
We can identify $p^{-1}(K_1)$ with  $K_1\times F,$
as  topological groups.
Observe that $K_1\times \{e\}$  is an open subgroup of the compact group
$\widetilde{K_0}$ and has therefore finite index in $\widetilde{K_0}$.
Hence, there exists a subgroup of finite index  $K_2$ of $K_1$ such that 
$K_2\times \{e\}$ is normal in $\widetilde{K_0}$.

Let $K$ be a compact open subgroup of $K_0$ 
contained in $K_2$ and  with the  properties  (i),  (ii), and (iii) as above.
Set 
 $$ L:= K\times \{e\} \quad\text{and} \quad  L^+:= L\cap \widetilde{P_0},\quad
L^-:= L\cap \underline{N_0^-}.
$$
Since $K$ is normal in $K_0$
 and  since $L$ is contained in the normal subgroup $K_2\times \{e\}$ of
 $\widetilde{K_0}=p^{-1}(K_0)$, it is clear that $L$ is normal in $\widetilde{K_0}$.

Moreover, as  $K= K^+ K^-$ for $K^+= K\cap P_0$ and $K^-= K\cap N_0^-,$
we have $L=L^+ L^-$. 
So, properties  (4.i) and (4.ii) are satisfied. Property  (4.iii) is also satisfied.
Indeed, let $i\in\{1,\dots, n\}$ and $a_i= p(\widetilde{a_i}).$  On the one hand, we have
$$
p(\widetilde{a_i}^{-1} L^+\widetilde{a_i})= a_i^{-1} K^+{a_i}\subset K^+=p(L^+)
$$
and 
$$
p(\widetilde{a_i} L^-\widetilde{a_i}^{-1})= a_i K^-{a_i}^{-1}\subset K^-= p(L^-).
$$
On the other hand, since $L$ is contained in $V\times \{e\},$
we have $$
\widetilde{a_i}^{-1} L^+\widetilde{a_i} \subset U\times \{e\}
\quad\text{and} \quad \widetilde{a_i}L^-\widetilde{a_i}^{-1}  \subset U\times \{e\}
$$
and the claim follows.
\end{proof}

It is known (see \cite[Theorem 15]{BernsteinNotes}) that every irreducible smooth 
representation of $G$ is admissible;
the following  proposition   extends this result to the central cover $\widetilde{G}.$
The proof depends on an analysis of cuspidal representations of $\widetilde{G}.$

Let $\widetilde P$ be a parabolic subgroup of $\widetilde{G}$
with Levi decomposition $\widetilde P=\widetilde M \underline{N},$
as defined above.

Let  $(\sigma, W)$ be a smooth representation of  $\widetilde{M}.$
Since $\widetilde P/\underline{N}= \widetilde M,$
the representation $\sigma$ extends  to a unique representation of $\widetilde{P}$ which is trivial on $\underline{N}$. One defines the induced representation $\Ind_{\widetilde P}^{\widetilde G}\sigma$ to be the right regular representation of $\widetilde{G}$ on the vector space $V$ of all 
locally constant functions $\widetilde{G}\to W$ with $f(pg)= \sigma(p) f(g)$ for all 
$g\in \widetilde{G}$ and $p\in \widetilde{P}.$

Let $(\pi, V)$ be a smooth representation of  $\widetilde{G}.$
Let $V(\underline{N})$ be the $\pi(\widetilde{M})$-invariant subspace generated 
by $\{\pi(n)v-v\mid n\in \underline{N},\, v\in V\}$; so, a smooth representation 
$\pi_{\underline{N}}$ of $\widetilde{M}$ is defined on $V_{\underline{N}}:= V/V(\underline{N})$.

An irreducible smooth representation $(\pi, V)$  of  $\widetilde{G}$ is  called \textbf{cuspidal} if 
 $V_{\underline{N}}=\{0\},$  for all unipotent radicals $\underline{N}$ of proper parabolic subgroups of $\widetilde{G}$.
\begin{proposition}
\label{Pro-AdmissibleCover} 
Every irreducible  smooth 
representation of $\widetilde{G}$ is admissible.
\end{proposition}
\begin{proof}
The proof is along the lines given for $G=\GG(\kk)$ 
in \cite{BernsteinNotes} or \cite{Renard}.

\vskip.2cm
\noindent
$\bullet$ \emph{First  step:} Let  $(\pi, V)$ be a smooth cuspidal representation
  of  $\widetilde{G}$. Then the matrix coefficients 
  of $\pi$ (that is, the functions $g\mapsto v^*(\pi(g)v)$,
  for $v\in V$ and $v^*\in V^*$) are compactly supported
  modulo the center  of $\widetilde{G}.$
  
  This  follows by an adaptation of the proof of the corresponding result for $G$ 
  as in  \cite[Theorem 14]{BernsteinNotes} or \cite[\S VI.2]{Renard}.

  \vskip.2cm
\noindent
$\bullet$ \emph{Second  step:} Let  $(\pi, V)$ be a smooth cuspidal  representation
  of  $\widetilde{G}$. Then $\pi$ is admissible.
  
  This is the same proof as in the case of $G$ (see \cite[Corollary p.53]{BernsteinNotes}
  or \cite[\S VI.2]{Renard}.

   \vskip.2cm
\noindent
$\bullet$ \emph{Third  step:} Let  $(\pi, V)$ be a smooth irreducible representation
  of  $\widetilde{G}$. Then there exists a   parabolic subgroup $\widetilde P=\widetilde M \underline{N}$ and a  cuspidal representation $\sigma$ of $\widetilde M$ such that 
 $\pi$ is isomorphic to a subrepresentation of  $\Ind_{\widetilde P}^{\widetilde G}\sigma$.
 
 The proof is identical to the proof  in the case of $G$ as in  \cite[Lemma 17]{BernsteinNotes} or \cite[Corollaire p.205]{Renard}. 

   \vskip.2cm
\noindent
$\bullet$ \emph{Fourth  step:}  Let  $(\pi, V)$ be a smooth irreducible representation
  of  $\widetilde{G}$. Then $\pi$ is admissible.

  By the third step, we may assume that $\pi$ is a subrepresentation of 
  $\Ind_{\widetilde P}^{\widetilde G}\sigma$ for some parabolic subgroup $\widetilde P=\widetilde M \underline{N}$ and a  cuspidal representation $\sigma$ of $\widetilde M$.
  By the second step, $\sigma$ is admissible.  Since $\widetilde G/\widetilde P$ is compact, it is easily seen that $\Ind_{\widetilde P}^{\widetilde G}\sigma$ is admissible (see
 \cite[Lemma III.2.3]{Renard}). It follows that $\pi$ is admissible.

\end{proof}

\begin{corollary}
\label{Cor-Pro-Covering}
Let   $\GG$ be a reductive linear algebraic group over a non archimedean local field 
$\kk.$  Let  $\widetilde{G}$ be a  central finite cover of $G= \GG(\kk)$ 
 satisfying Condition $(*)$ as above. Then   $\widetilde{G}$ is  trace class.
 \end{corollary}
 \begin{proof}
 Let $L$ be a compact open subgroup of $\widetilde{G}$.
 Proposition~\ref{Pro-Covering} shows that Assertion A from \cite{Bernstein}
is satisfied for the totally disconnected locally compact group $\widetilde{G}$.
Hence, there exists an integer $N(L)\geq 1$ such that 
$\dim V^L \leq N(L)$ for every \emph{admissible}\ smooth representation
$(\pi, V)$ of $\widetilde{G}.$ Since, by Proposition~\ref{Pro-AdmissibleCover},
 every smooth representation of  $\widetilde{G}$ is admissible, 
 it follows that $\widetilde{G}$ has a uniformly admissible smooth dual.
Proposition~\ref{Pro-AdmissibleCompactSubgroups}.(ii) shows that
 $\widetilde{G}$  is therefore trace class.

\end{proof}

The metapletic group $\widetilde{Sp_{2n}(\kk)}$ satisfies Condition $(*)$ as above
(see \cite[Corollaire p. 38]{MVW}). The following result is therefore a direct consequence
of Corollary~\ref{Cor-Pro-Covering}.

\begin{corollary}
\label{Cor-Covering}
Let $\GG$ be a reductive algebraic subgroup of $Sp_{2n}$
and let $G= \GG(\kk).$ 
Let $\widetilde G$ be the inverse image of $G$ in  $\widetilde{Sp_{2n}(\kk)}$.
Then $\widetilde G$ is a trace class group.
\end{corollary}

\section{Proofs}
\subsection{Proof of Theorem~\ref{MainTheorem}} 
\label{S2}

Let $\GG^0$ be the connected component of $\GG.$
Since $\GG^0(\kk)$ is a normal subgroup of finite index in $\GG(\kk)$, 
in view of Proposition~\ref{Prop-FiniteIndex}, it suffices to prove the claim
when $\GG$ is connected. 

Let $\UU$ be the unipotent radical  of $\GG$ and $\mathfrak{u}$ its Lie algebra.
Through a series of reduction steps, we will be lead eventually to the case where 
$\UU$ is a Heisenberg group.  

If $\mathfrak{u}=\{0\}$, then  $\GG$ is reductive and the
claim follows from \cite{Dixmier} and \cite{Bernstein}.

\vskip.2cm
\noindent
$\bullet$ \emph{First reduction step:} we  assume from now on that $\GG$ is connected and that {$\mathfrak{u}\neq \{0\}$.}

We proceed by induction on $\dim \GG.$ So, assume that $\dim \GG>0$ and the claim is proved for every connected algebraic group defined over $\kk$ with dimension strictly  smaller than $\dim \GG.$ 

Let $\rho$ be a factor representation  of $G:=\GG(\kk),$ fixed in the sequel.
We have to show that $\rho$ is of type I.

Let  $\mathfrak{m}$ be a non-zero $\GG$-invariant abelian ideal of 
 {$\mathfrak{u}$. (An} example of such an ideal is the center of $\mathfrak{u}$.)
 Then $\MM=\exp(\mathfrak{m})$ is a non-trivial  abelian algebraic normal subgroup of $\GG$ contained in $\UU$.   By Lemma~\ref{Lem-RegEmbedded}, $M:=\MM(\kk)$ is regularly embedded in $G.$
Hence, by Theorem~\ref{Mackey}.(ii), there exists $\chi\in \widehat{M}$ and 
a factor  representation $\pi$ of $G_\chi$ such that $\pi|_{M}$ is a multiple of $\chi$
and such that $\rho= \Ind_{G_\chi}^G \pi.$

 Recall that  $\GG$ acts $\kk$-rationally by the co-adjoint action of $\GG$ on 
$\mathfrak{m}^*$ and that $\chi$ is defined by a unique linear functional 
$f_\mathfrak{m}\in \mathfrak{m}^*$  (see Subsection~\ref{SS:Actions}).

Assume that either $\dim {\mathfrak{m}}\geq 2$ or $\dim{\mathfrak{m}}=1$
and $f_\mathfrak{m}=0$.
 In both cases,  there exists a non-zero subspace ${\mathfrak{k}}$ of $\mathfrak{m}$ 
contained in $\ker (f_{\mathfrak{m}}).$ Let $\KK$ be the corresponding  
algebraic normal subgroup of $\GG.$
Then $\rho$ factorizes through the group of $\kk$-points
of the connected algebraic group $\GG/\KK$ and $\dim \GG/\KK<\dim \GG.$
Hence, $\rho$ is of type I, by the induction hypothesis.

\vskip.2cm
\noindent
$\bullet$ \emph{Second reduction step:}  we  assume from now on that
$\mathfrak{u}$ contains  no $\GG$-invariant ideal $\mathfrak{m}$ with  $\dim {\mathfrak{m}}\geq 2$ or 
with $\dim {\mathfrak{m}}=1$
and such that $f_\mathfrak{m}=0$.
In particular, we have $\dim \mathfrak{z}=1$ and $f_\mathfrak{z}\neq 0,$
where $\mathfrak{z}$ is
 the center of $\mathfrak{u}.$

Set $\ZZ=\exp(\mathfrak{z})$  and set $Z=\ZZ(\kk)$.
Let  $\chi\in \widehat{Z}$ be the character corresponding to $f_\mathfrak{z}$
and let $\pi$  be a factor  representation of $G_\chi$ such that $\pi|_{Z}$ is a multiple of $\chi$ and such that $\rho= \Ind_{G_\chi}^G \pi.$

Assume that $\dim \GG_\chi<\dim \GG.$
Then, by induction hypothesis, the group $\GG^0_\chi(\kk)$ of the 
$\kk$-points of the connected component of $\GG_\chi$ is of type I. 
Since $\GG^0_\chi(\kk)$ has finite index in $G_\chi$, it follows
from Proposition~\ref{Prop-FiniteIndex} that $G_\chi$ is of type I. Hence, $\rho$ is of type I 
by Theorem~\ref{Mackey}.(i). 

\vskip.2cm
\noindent
$\bullet$ \emph{Third reduction step:}  we  assume from now on that
$ \dim \GG_\chi=\dim \GG$ for $\chi\in \widehat{Z}$  as above. Then $\GG_\chi$ has finite index in  $\GG$ and hence   $\GG_\chi=\GG$, since $\GG$ is connected.
So, $G_\chi=G$ and in  particular $\rho=\pi.$

Observe that, since $\chi\neq 1_Z$ and since $\dim \ZZ=1,$ it follows
that $\GG$ acts trivially on $\ZZ.$ 
We claim that $\mathfrak{u}$ is not abelian. 
Indeed, otherwise, we would have $\mathfrak{u}=\mathfrak{z}$
and $\GG$ would act trivially on $\mathfrak{u};$ so, $\GG$  would be reductive, in contradiction to the  first  reduction step.

We claim that $\dim [\mathfrak{u},\mathfrak{u}]= 1.$
Indeed, assume, by contradiction,  that $\dim [\mathfrak{u},\mathfrak{u}]\geq 2.$
Then, setting $\mathfrak{u}':=[\mathfrak{u},\mathfrak{u}],$  there exist  ideals $\mathfrak{u}_1, \mathfrak{u}_2$ in $\mathfrak{u}$ of dimension $1$ and $2$
such that $\mathfrak{u}_1 \subset \mathfrak{u}_2 \subset \mathfrak{u}'$ {and $[\mathfrak{u},\mathfrak{u}_2]\subset \mathfrak{u}_1$.}
Since
$$
[\mathfrak{u'},\mathfrak{u}_2]\subset [[\mathfrak{u},\mathfrak{u}_2],\mathfrak{u}]+
[[\mathfrak{u}_2,\mathfrak{u}],\mathfrak{u}] {\subset}[\mathfrak{u}_1,\mathfrak{u}]=0,
$$
the center $\mathfrak{z}'$ of $\mathfrak{u'}$ is of dimension $\geq 2.$
As $\mathfrak{z}'$ is a characteristic ideal, $\mathfrak{z}'$ is $\GG$-invariant
and this is a contradiction to the  second reduction step.

 In view  of Lemma~\ref{Lem-Heisenberg}, we can state our last reduction step.
 
\vskip.2cm
\noindent
$\bullet$ \emph{Fourth  reduction step:}  we  assume from now on that
 $\UU$ is the Heisenberg group $H_{2n+1}$ for some $n\geq 1$ and that 
$\chi\neq 1_Z.$

By Theorem~\ref{Theo-StoneVN-Weil}.(i),
there exists a  unique irreducible representation $\pi_\chi$ of $H_{2n+1}(\kk)$
such that $\pi_\chi|_{Z}$ is a multiple of $\chi$.

Let $\LL$ be a Levi subgroup 
of $\GG$ defined over $\kk.$ So, $\LL$ is reductive and  $\GG= \LL\ltimes \UU.$
By Proposition~\ref{Pro-AutHeisenberg}, {the action of $G$ on 
$H_{2n+1}(\kk)$ by conjugation gives rise to a continuous homomorphism
$\vfi: G\to Sp_{2n}(\kk)\ltimes H_{2n+1}(\kk).$}

The representation $\pi_\chi$ of $H_{2n+1}(\kk)$ extends to an
$\omega\circ (p\times p)$-repre\-sen\-tation $\widetilde{\pi_\chi}$ of $Sp_{2n}(\kk)\ltimes H_{2n+1}(\kk)$
for some $\omega \in Z^2(Sp_{2n}(\kk), \mathbf{T}),$
where $p:Sp_{2n}(\kk)\ltimes H_{2n+1}(\kk)\to Sp_{2n}(\kk)$ is the canonical projection.
By Theorem~\ref{Theo-StoneVN-Weil}.{(ii)},
one can choose $\omega$ so that the image of $\omega$ 
is $\{\pm 1\}$ and the 
corresponding central extension of $Sp_{2n}(\kk)$
is the metaplectic group $\widetilde{Sp_{2n}(\kk)}.$
Setting 
$$\omega':=\omega\circ (p\times p) \circ (\vfi\times\vfi)\in Z^2(G, \mathbf{T})$$
and $\widetilde{\pi}':=\widetilde{\pi_\chi} \circ \vfi,$
we have that  $\widetilde{\pi}'$ is an $\omega'$-representation
of $G$ which extends $\pi_\chi$. 

Let $\sigma$ be a factor $\omega'$-representation
of $L$ lifted to $G.$ (Observe that $\omega'^{-1}= \omega'$.)
Let $ L^{\omega'}= \{\pm 1\}\times L$ be the  central extension of 
 $L$ defined by $\omega'$.

  We claim that $\widetilde{L}$ satisfies Condition $(*)$ as before Proposition~\ref{Pro-Covering}. Indeed,
let  $x,y\in L$ be such that $xy=yx$; then $\vfi(x)\vfi(y)= \vfi(y) \vfi(x)$
 and hence  
 $$\omega'(x,y)=\omega(\vfi(x),\vfi(y))=\omega(\vfi(y),\vfi(x))=\omega'(y,x),$$
  since Condition $(*)$ holds for  $\widetilde{Sp_{2n}(\kk)}.$
  This proves the claim.  
  
 It follows from Corollary~\ref{Cor-Pro-Covering} that  $\widetilde{L}$
 is of type I. 
Theorem~\ref{Mackey}.(iii) then implies that 
the factor representation $\rho$ of $G$ is of type $I.$

\subsection{Proof of Theorem~\ref{SecondTheorem}} 
\label{S2-SS1}

As mentioned before, the fact that  
(i) implies  (ii) was already shown by Lipsman in \cite[Theorem 3.1 and Lemma 4.1]{Lipsman-CCR}. 
\magenta{}

We claim that (ii) implies  (i).
Indeed,  let $\MM$ be as in Theorem~\ref{SecondTheorem}. 
Then $\MM$ is an algebraic  normal subgroup of the reductive group $\LL$
and is therefore reductive as well (see \cite[Corollary in \S 14.2]{Borel}).
 Since $\MM(\kk)$ commutes with all elements in $\UU(\kk)$, it follows that 
$$N:=\MM(\kk)\UU(\kk)\cong \MM(\kk)\times \UU(\kk)$$
 is a direct product of groups.
Now,  $\MM(\kk)$ and $\UU(\kk)$ are CCR (even trace class) by the results mentioned in the introduction. Hence,  $N$ is CCR as well (and by  \cite[Proposition 1.9]{Deitmar-vanDijk} even trace class).
Since $\GG(\kk)/N\cong \LL(\kk)/\MM(\kk)$ is compact, it follows from 
 a general result  (see \cite[Proposition 4.3]{Schoch}) that  $\GG(\kk)$ is CCR.

It remains to show that (ii) implies (iii).

If  $\kk$ is non-archimedean then $G=G(\kk)$ is totally disconnected and CCR and 
it follows from Corollary \ref{cor-CCR-trace-class} that $G$ is trace class.
The proof in the archimedean case is more involved.

Assume that $\kk=\RRR$ and that  (ii) holds. 
Let 
$$N:=\MM(\RRR)\UU(\RRR)\cong \MM(\RRR)\times \UU(\RRR)$$ be as above. We already observed above
that $N$ is trace class and that $G/N$ is compact, where $G=\GG(\RRR)$. 
Note that $N$ is unimodular, since $\MM(\RRR)$ and $\UU(\RRR)$ are unimodular.
Since, moreover, $G/N$ is compact,
 it follows  from Weil's formula
$$\int_G f(g)\, dg=\int_{G/N} \int_N f(gn)\,dn\, d \dot{g},\quad f\in C_c(G)$$
that $G$ is unimodular as well.

Since subrepresentations  of trace class representations are trace class
and  in view of Corollary \ref{lem-subrep}, it suffices to show
that, for every $\pi\in \widehat{N},$ the induced representation $\rho=\Ind_N^G\pi$ is a (not necessarily irreducible) 
trace class representation of $G$, i.e., the operator $\rho(f)$ is a trace-class operator for all $f\in C_c^{\infty}(G)$.

Recall  that, since $G/N$ is compact and both $G$ and $N$ are unimodular,
the Hilbert space $\H_\rho$ for the induced representation $\rho=\Ind_N^G\pi$ is a completion of the 
vector space
\begin{equation}\label{eq-indspace}
\mathcal F_\rho=\{\xi\in C(G,\H_\pi):  \xi(gn)=\pi(n^{-1})\xi(g)\quad \forall g\in G\; \forall n\in N\},
\end{equation}
with inner product given by
$\langle \xi,\eta\rangle=\int_{G/N} \langle \xi(g), \eta(g)\rangle \, d\dot{g}.$
The induced representation $\rho=\Ind_N^G\pi$ acts on the dense subspace $\F_\rho$ by 
the formula $(\rho(s)\xi)(g)=\xi(s^{-1} g).$

Now let $f\in C_c^{\infty}(G)$. Following some ideas in the proof of
\cite[Proposition 4.2]{Schoch}, we compute for $\xi\in \mathcal F_\rho$:
\begin{align*}
\big(\rho(f)\xi\big)(g)&=\int_G f(s)\xi(s^{-1}g)\, ds
\stackrel{s\mapsto gs^{-1}}{=} \int_G f(gs^{-1})\xi(s)\, ds\\
&=\int_{G/N}\int_N f(gn^{-1}s^{-1})\xi(sn)\, dn \, d\dot{s}\\
&\stackrel{n\mapsto n^{-1}}{=} \int_{G/N}\int_N f(gns^{-1})\xi(sn^{-1})\, dn \, d\dot{s}\\
&= \int_{G/N} \int_Nf(gns^{-1})\pi(n)\xi(s)\, dn \, d\dot{s}\\
&=\int_{G/N} k_f(g, s) \xi(s)\, d\dot{s},
\end{align*}
where $k_f:G\times G\to \mathcal B(\H_\pi)$ is given by
$k_f(g,s)=\pi(\varphi_f(g,s))$ and  
$\varphi_f:G\times G\to C_c^{\infty}(N)$ is given by $\varphi_f(g,s)(n):=f(gns^{-1})$. 

Since $\pi$ is trace class, 
the map $k_f=\pi\circ \varphi_f$ takes its  values in the space of
trace-class operators.
As trace class operators are  Hilbert-Schmidt operators, we may regard 
$k_f$ as a map into the 
set $\mathcal{H}S(\H_\pi)$ of Hilbert-Schmidt operators on $\H_\pi$.

We claim that the map $G\times G\to \|k_f(g,s)\|_{\mathrm{HS}}$ is continuous, where  $\|\cdot\|_{\mathrm{HS}}$
denotes the  Hilbert-Schmidt norm.
This will follow from \cite[Proposition 1.4]{Deitmar-vanDijk} as soon as we have shown that the map
$$\varphi_f:G\times G\to C_c^{\infty}(N)$$ is continuous,
where $C_c^\infty(N)$ is equipped with the Fr\'echet space  topology as 
introduced in \cite[Definition 1.1]{Deitmar-vanDijk}. But this is a consequence of the  continuity of  the map   $C_c^\infty(G)\to C_c^\infty(N),\varphi\mapsto \varphi|_N$  and of 
the map $G\times G\to C_c^\infty(G), (g,s)\mapsto L_sR_gf $, where $L_s$  (resp. $R_g$) 
denote left (resp. right) translation by $s$ (resp. $g$).

For $k,m\in N$, we have
\begin{align*}
k_f(gk,sm)&=\int_N f(gknm^{-1}s^{-1})\pi(n)\, dn\\
&\stackrel{n\mapsto k^{-1}nm}{=}  \int_N f(gns^{-1})\pi(k^{-1}nm)\, dn\\
&=\pi(k^{-1})k_f(g,s)\pi(m).
\end{align*}
It follows that  the Hilbert-Schmidt norm of $k_f(g,s)$ is constant on $N$-cosets 
in both variables.  As $G/N$ is compact, this implies that 
the continuous map $(g,s)\mapsto \|k_f(g,s)\|_{\HS}$ is bounded.

We want to conclude from this that $\rho(f)$ is a Hilbert-Schmidt operator. For this, let
$c:G/N\to G$ be a Borel section for the quotient map $p:G\to G/N$. Let
$V: \H_\rho\to L^2(G/N, \H_\pi)$ be defined on $\F_\rho$ by $\xi\mapsto V\xi:=\xi\circ c$.
Then it is straightforward to check  that $V$ is a unitary operator which intertwines $\rho(f)$ with the integral operator
$K_f:L^2(G/N,\H_\pi)\to L^2(G/N, \H_\pi)$ given by
 $$(K_f\xi)(\dot g)=\int_{G/N} \tilde{k}_f(\dot g, \dot s)\xi(\dot s)\, d\dot s,$$
 where $\tilde{k}_f(\dot g, \dot s):=k_f(c(\dot g), c(\dot s))$ for $\dot g, \dot s\in G/N$.
 
It suffices to prove that $K_f$ is Hilbert-Schmidt. For this,
let $\{e_i: i\in I\}$ and $\{v_j:j\in J\}$ be orthonormal bases 
of $L^2(G/N)$ and $\H_\pi$, respectively. 
Observe that $I$ and $J$ are at most countable,
since $L^2(G/N)$ and $\H_\pi$ are separable, by the second countability of
$G.$

For a pair $(j,l)\in J\times J,$ 
denote by 
$$k_{jl}(\cdot, \cdot):=\langle \tilde{k}_f(\cdot, \cdot) v_j, v_l\rangle$$ 
the $(j,l)$-th matrix coefficient of $\tilde k_f$ with respect to $\{v_j:j\in J\}$. 
Identifying elements $g\in G$ with their images in $G/N$ in the following formulas, we then get
\begin{equation}\label{eq-1}
\langle \tilde k_f(g,s) v_j, \tilde k_f(g,t) v_j\rangle =\sum_{l} k_{jl}(g,s)\overline{k_{jl}(g,t)}
\end{equation}
 for $g,s,t\in G$. Using this and the fact that $\{\bar{e_i}: i\in I\}$ is also an orthonormal basis of $L^2(G/N)$, we compute
 \begin{align*}
 \|K_f\|_{\HS}^2&=\sum_{i,j}\langle K_f(e_i\otimes v_j), K_f (e_i\otimes v_j)\rangle\\
 &=\sum_{i,j}\int_{G/N\times G/N\times G/N} \left\langle \tilde k_f(g,  s) e_i(s)v_j,  \tilde k_f( g,  t) e_i(t)v_j\right\rangle \,d( g, s,  t)\\
 &\stackrel{(\ref{eq-1})}{=}\sum_{i,j}\int_{G/N\times G/N\times G/N} e_i(s) \overline{e_i(t)} \sum_l k_{jl}(g,s)\overline{k_{jl}(g,t)} \, d(g,s,t)\\
 &=\sum_{i,j,l} \int_{G/N}\left\langle k_{jl}(g, \cdot), \bar{e}_i\right\rangle \left\langle \bar{e_i}, k_{jl}(g, \cdot )\right\rangle \,dg\\
 &=\sum_{j,l} \int_{G/N} \|k_{jl}(g, \cdot)\|_2^2\, dg\\
 &=\sum_{j,l} \|k_{jl}\|_2^2 
 =\sum_{j,l} \int_{G/N\times G/N} | k_{jl}(g,s)|^2\, d(g,s)\\
 &\stackrel{(\ref{eq-1})}{=}\int_{G/N\times G/N} \sum_j\left\langle \tilde k_f(g,s) v_j, \tilde k_f(g,s) v_j\right\rangle \, d(g, s)\\
 &=\int_{G/N\times G/N} \| \tilde k_f(g,s)\|^2_{\HS}\, d(g,s)
 <\infty.
 \end{align*}
It follows that $\rho(f)$ is a Hilbert-Schmidt operator for all $f\in C_c^\infty(G)$, and since 
$C_c^\infty(G)=C_c^\infty(G)*C_c^\infty(G)$ by \cite[Theorem 3.1]{DM}, this implies that $\rho(f)$ is trace class for all $f\in C_c^\infty(G).$
This finishes the proof.

\end{document}